\documentclass[preprint,12pt]{elsarticle}
\usepackage{tikz}
\usepackage{amsthm,amsfonts,amssymb,amscd,amsmath,enumerate,verbatim,calc,graphicx,geometry}
\usepackage[all]{xy}
\newtheorem{theorem}{Theorem}[section]

\newtheorem{proposition}[theorem]{Proposition}
\newtheorem{corollary}[theorem]{Corollary}
\theoremstyle{definition}
\theoremstyle{definitions}
\newtheorem{definition}[theorem]{Definition}

\newtheorem{remark}[theorem]{Remark}
\newtheorem{example}[theorem]{Example}
\theoremstyle{notations}

\theoremstyle{remarks}

\journal{}

\begin{document}

\begin{frontmatter}



\title{On Strong Small Loop Transfer Spaces Relative to Subgroups of Fundamental Groups}


\author[]{S.Z. Pashaei}
\ead{Pashaei.seyyedzeynal@stu.um.ac.ir}
\author[]{B. Mashayekhy\corref{cor1}}
\ead{bmashf@um.ac.ir }
\author[]{M. Abdullahi Rashid}
\ead{mbinev@stu.um.ac.ir }

\address{Department of Pure Mathematics, Center of Excellence in Analysis on Algebraic Structures, Ferdowsi University of Mashhad,\\
P.O.Box 1159-91775, Mashhad, Iran.}
\cortext[cor1]{Corresponding author}
\begin{abstract}
Let $H$ be a subgroup of the fundamental group $\pi_{1}(X,x_{0})$. By extending the concept of strong SLT space to a relative version with respect to $H$, strong $H$-SLT space, first, we investigate the existence of a covering map for strong $H$-SLT spaces. Moreover, we show that a semicovering map is a covering map in the presence of strong $H$-SLT property. Second, we present conditions under which the whisker topology agrees with the lasso topology on $\widetilde{X}_{H}$. Also, we study the relationship between open subsets of $\pi_{1}^{wh}(X,x_{0})$ and $\pi_{1}^{l}(X,x_{0})$. Finally, we give some examples to justify the definition and study of strong $H$-SLT spaces.
\end{abstract}


\end{frontmatter}



\section{Introduction and Motivation}

Throughout this article, we consider a path connected topological space $X$ with a base point $x_{0}\in{X}$. Given a pointed topological space $(X,x_{0})$, we denote the set of all paths in $X$ starting at $x_{0}$ by $P(X,x_{0})$. Let $H$ be a subgroup of $\pi_{1}(X,x_{0})$ and $\widetilde{X}_{H}= P(X,x_{0})/ \sim$, where $\alpha \sim \beta$ if and only if $\alpha(1)=\beta(1)$ and $[\alpha\ast\beta^{-1}]\in{H}$. The equivalence class of $\alpha$ is denoted by $[\alpha]_{H}$.  We denote the constant path at $x_{0}$ by $c_{x_{0}}$. Note that $[\alpha]_{H} = [c_{x_{0}}]_{H}$ if and only if $[\alpha]\in{H}$. The map $p_{H}: \widetilde{X}_{H}\rightarrow X$ is defined to be the endpoint projection $p_{H}([\alpha]_{H})=\alpha(1)$. We denote $p_{e}$ and $\widetilde{X}_{e}$ instead of $p_{H}$ and $  \widetilde{X}_{H}$, respectively, when $H$ is the trivial subgroup.

There are three famous topologies on $\widetilde{X}_{H}$. One of them is the quotient topology induced by the compact-open topology on  $P(X,x_{0})$. We denote the space $\widetilde{X}_{H}$ equipped with this topology by $\widetilde{X}^{top}_{H}$. The second topology is the whisker topology which was introduced by Spanier \cite[Theorem 2.5.13]{Spanier} and named by Brodskiy et al. \cite{Bcovering}, as follows.

 \begin{definition}
 For any pointed topological space $(X, x_{0})$ the whisker topology on the set $\widetilde{X}_{H}$ is defined by the collection of all the following sets as a basis
 $$B_{H}([\alpha]_{H}, U)=\lbrace [\beta]_{H}\in{\widetilde{X}_{H}} \ \vert \ \beta \simeq \alpha\ast\lambda \ for \ some \ \lambda :I \rightarrow U, \ \lambda(0)=\alpha(1)\rbrace,$$
 where $[\alpha]_{H}\in{\widetilde{X}_{H}} $ and $U$ is an open neighborhood of $\alpha(1)$. In the case $H=1$, we denote the basis elements of $\widetilde{X}_{e}$ by $B([\alpha], U)$.
 \end{definition}

The Spanier group $\pi(\mathcal{U}, x)$ \cite{Spanier} with respect to an open cover
$\mathcal{U} = \lbrace U_{i} \ | \ i \in{I} \rbrace$ is defined to be the subgroup of $\pi_{1}(X, x)$ which contains all homotopy classes having representatives of the type
$\prod_{j=1}^{n} \alpha_{j}\beta_{j}\alpha^{-1}_{j}$, where $\alpha_{j}$'s are arbitrary paths starting at $x$ and each $\beta_{j}$ is a loop inside one of the
open sets $U_{j}\in{\mathcal{U}}$.

The next topology is lasso topology which has been introduced and studied in \cite{Bcovering}.

\begin{definition}\label{1.2}
 For any pointed topological space $(X, x_{0})$ the lasso topology on the set $\widetilde{X}_{H}$ is defined by the collection of all the following sets as a basis

$$B_{H}([\alpha]_{H},\mathcal{U}, U)=\lbrace [\beta]_{H}\in{\widetilde{X}_{H}} \ \vert \ \beta \simeq \alpha\ast\gamma\ast\delta \ for \ some \ [\gamma]\in{\pi(\mathcal{U}, \alpha(1))} $$
$$\ \ \ \ \  \ \  and \ for \ some \ \delta:I\rightarrow U, \ \delta(0)=\alpha(1)\rbrace,$$

where $[\alpha]_{H}\in{\widetilde{X}_{H}}$, $\mathcal{U}$ is an open cover of $X$ and $U\in{\mathcal{U}}$ is an open neighborhood of $\alpha(1)$.  In the case $H=1$, we denote the basis elements of $\widetilde{X}_{e}$ by $B([\alpha],\mathcal{U}, U)$.
\end{definition}

We denote the space $\widetilde{X}_{H}$ equipped with whisker and lasso topologies by $\widetilde{X}^{wh}_{H}$ and $\widetilde{X}^{l}_{H}$, respectively. Moreover, we denote $\widetilde{X}^{top}_{e}$, $\widetilde{X}^{wh}_{e}$ and $\widetilde{X}^{l}_{e}$ instead of $\widetilde{X}^{top}_{H}$, $\widetilde{X}^{wh}_{H}$ and $\widetilde{X}^{l}_{H}$, respectively, when $H$ is the trivial subgroup. Note that $\pi_{1}^{qtop}(X,x_{0})$ and $\pi_{1}^{wh}(X,x_{0})$ and $\pi_{1}^{l}(X,x_{0})$ can be considered as subspaces of $\widetilde{X}^{top}_{e}$, $\widetilde{X}^{wh}_{e}$ and $\widetilde{X}^{l}_{e}$, respectively. The relation between these three different topologies are as follows, when $X$ is a connected, locally path connected space (see \cite{VZcom}).

\begin{center}
``$\widetilde{X}^{wh}_{e}$ is finer than $\widetilde{X}^{top}_{e}$'' and  `` $\widetilde{X}^{top}_{e}$ is finer than  $\widetilde{X}^{l}_{e}$''
\end{center}
Note that similar statements to the above hold for $\widetilde{X}_{H}$ when $H$ is a nontrivial subgroup (see \cite{BrazG}).

Small loop transfer (SLT for short) spaces were introduced for the first time by Brodskiy et al. \cite[Definition 4.7]{BroU}. The main motivation of the definition of SLT spaces is to determine the condition for coincidence of the compact-open topology and the whisker topology on $\widetilde{X}_{e}$. Indeed, Brodskiy et al. \cite[Theorems 4.11, 4.12]{BroU} proved that a locally path connected space $X$ is an SLT space if and only if for every $x\in{X}$, $\widetilde{X}^{top}_{e}=\widetilde{X}^{wh}_{e}$. Also, they defined a strong version of this notion, strong SLT space \cite[Definition 4.18]{BroU}, and showed that a path connected space $X$ is a strong SLT space if and only if for every $x\in{X}$, $\widetilde{X}^{l}_{e}=\widetilde{X}^{wh}_{e}$ \cite[Proposition 4.19]{BroU}. Moreover, Pashaei et al. \cite{Pasha} introduced and studied SLT spaces with respect to a subgroups $H$ of $\pi_{1}(X,x_{0})$ ($H$-SLT for short) at point $x_{0}$, and using this notion, presented a condition for the coincidence of the whisker and the compact-open topology on $\widetilde{X}_{H}$. In this paper by introducing a relative version of strong small loop transfer spaces with respect to a subgroup $H$ of $\pi_{1}(X,x_{0})$ at a point $x_{0}$ (strong $H$-SLT at $x_{0}$), we are going to determine when the whisker and the lasso topologies on $\widetilde{X}_{H}$ are identical. Also,
 we study the relationship between covering and semicovering spaces of strong $H$-SLT spaces at a point in $X$.

\begin{definition}
Let $H$ be a subgroup of $\pi_{1}(X,x_{0}) $. A topological space $X$ is called strong $H$-small loop transfer (strong $H$-SLT for short) space at $x_{0}$ if for every $x\in X$ and for every open neighborhood $U$ of $X$ containing $x_{0}$ there is an open neighborhood $V$ containing $x$ such that for every loop $\beta: I\rightarrow V$ based at $x$ and  for every path $\alpha: I\rightarrow X$ from $x_{0}$ to $x$ there is a loop $\lambda: I\rightarrow U$ based at $x_{0}$ such that, $[\alpha \ast \beta \ast \alpha^{-1}]_{H}=[\lambda]_{H}$. Also, $X$ is called a strong $H$-SLT space if for every $x\in{X}$ and for every path $\delta$ from $x_{0}$ to $x$, $X$ is a strong $[\delta^{-1}H\delta]$-SLT space at $x$. Note that if $H$ is the trivial subgroup, then a strong $H$-SLT space is a strong SLT space.
\end{definition}

 It is well known that covering spaces of a path connected, locally path connected and semilocally simply connected space $X$ are classified by subgroups of the fundamental group $\pi_{1}(X,x_{0})$. When $X$ has more complicated local structure, there need not be a simply connected cover corresponding to the trivial subgroup. Many people have attempted to extend the covering-theoretic approach to more general spaces. A common approach is to designate those properties of a covering map which are assumed important. One of them is semicoverings \cite{BrazS} which are defined to be local homeomorphisms with continuous lifting of paths and homotopies which are related to topological group structures on fundamental groups \cite{BrazO, FischerC}. The other one is generalized universal coverings which were introduced by Fischer and Zastrow \cite{Zastrow} and provide combinatorial information about fundamental groups of spaces which are not semilocally simply connected such as the Hawaiian earring, the Menger curve, and the Sierpinski carpet.

The following theorem determines the existence of coverings for locally path connected spaces via Spanier groups \cite[Theorem 2.5.13]{Spanier}.

\begin{theorem}\label{1.4}
Let $X$ be connected, locally path connected and $H\leq \pi_{1}(X,x_{0})$. Then there exists a covering map $p:\widetilde{X}\rightarrow X$ with $p_{\ast}\pi_{1}(\widetilde{X},\tilde{x}_{0})=H$ if and only if   there is an open cover $\mathcal{U}$ of $X$ such that $\pi(\mathcal{U}, x_{0})\leq H$.
\end{theorem}

Brazas \cite[Theorem 4.8]{BrazO} and Torabi et al. \cite[Theorem 2.1]{TorabiS} have addressed the existence of covering spaces of locally path connected via algebraic and topological structures of fundamental groups. For instance, in \cite[Theorem 2.1]{TorabiS} it was shown that any open normal subgroup in $\pi_{1}^{qtop}(X,x_{0})$ contains a Spanier group. In Section 2, we show the existence of covering spaces of a strong $H$-SLT space at point $x_{0}$ provided that $H$ is open in $\pi_{1}^{qtop}(X,x_{0})$. In fact, we prove that if $K$ is an open subgroup in $\pi_{1}^{qtop}(X,x_{0})$ containing $H$, then there is an open cover $\mathcal{U}$ of $X$ such that $\pi(\mathcal{U}, x_{0})$ is contained in $K$ when $X$ is a strong $H$-SLT space at $x_{0}$ (see Proposition \ref{2.1}).

It is obvious that every covering map $p:\widetilde{X} \rightarrow X$ is a semicovering map but not vice versa \cite{FischerC}. Brazas showed that if $X$ is a connected, locally path connected and semilocally simply connected space, then these two concepts are the same (see \cite[Corollary 7.2]{BrazS}). Moreover, Torabi et al. \cite[Theorem 4.4]{TorabiS} proved this fact for connected, locally path connected and semilocally small generated spaces. In Corollary \ref{2.3}, we show that if $p:(\widetilde{X},\tilde{x}_{0}) \rightarrow (X,x_{0})$ is a semicovering map with $p_{\ast}\pi_{1}(\widetilde{X},\tilde{x}_{0})=H\leq \pi_{1}(X,x_{0})$, then $p$ is a covering map when $X$ is a connected, locally path connected and strong $H$-SLT space at $x_{0}$. Consequently, we can show that every semicovering map is a covering map in strong SLT spaces at $x_{0}$. Recall that Brazas introduced the notion of $\textbf{lpc}_{0}$-covering maps in terms of unique lifting property \cite[Definition 5.3]{BrazG}. Note that $\textbf{lpc}_{0}$-covering maps were inspired by the concept of generalized universal covering maps which introduced by Fischer and Zastrow \cite{Zastrow}. By the definition, it can be easily seen that every covering map is an $\textbf{lpc}_{0}$-covering map. Since the fibers of a covering map are discrete, Example 4.15 in \cite{Zastrow} implies that an $\textbf{lpc}_{0}$-covering map is not necessarily a covering map. In Proposition \ref{2.5}, we prove that if $X$ is a strong $H$-SLT space at $x_{0}$, then an $\textbf{lpc}_{0}$-covering map $p:(\widetilde{X},\tilde{x}_{0}) \rightarrow (X,x_{0})$ with $p_{\ast}\pi_{1}(\widetilde{X},\tilde{x}_{0})=H$ is a covering map when the fiber $p^{-1}(x_{0})$ is finite. Finally, we address the relationship between some famous subgroups of $\pi_{1}(X,x_{0})$ in strong SLT and SLT spaces at $x_{0}$.

The aim of Section 3 is to clarify the relationship between the whisker topology and the lasso topology on $\widetilde{X}_{H}$. We show that these two topologies on $\widetilde{X}_{H}$ are identical if and only if the space $X$ is a strong $H$-SLT space when $H$ is a normal subgroup (see Theorem \ref{3.2}). Moreover, we show that if $X$ is strong $H$-SLT at $x_{0}$, then all open subsets of $\pi_{1}^{wh}(X,x_{0})$ and $\pi_{1}^{l}(X,x_{0})$  containing normal subgroup $H$ are the same. In the case that $H$ is not normal, we show that open subgroups of $\pi_{1}^{wh}(X,x_{0})$ and $\pi_{1}^{l}(X,x_{0})$ containing $H$ are the same (see Proposition \ref{3.6}). However, we prove that if $X$ is a strong $H$-SLT space at $x_{0}$, then closed normal subgroups of $\pi_{1}^{wh}(X,x_{0})$ and $\pi_{1}^{l}(X,x_{0})$ containing $H$ are the same. In Corollary \ref{3.11}, we show that a semicovering map can transfer the property of being strong SLT from its codomain to its domain. Finally, in order to justify the definition of strong $H$-SLT spaces, we give an example of an strong $H$-SLT space which is not strong SLT and consequently, it is not semilocally simply connected (see Example \ref{3.13}). Also, we give an example to show that some results of the paper do not necessarily hold, for instance Proposition \ref{3.6}, for $H$-SLT spaces at $x_{0}$ (see Example \ref{3.14}).


\section{Relationship Between Strong SLT Spaces and Covering Maps}

Since existence of covering maps have a significant relation with Spanier groups (see Theorem \ref{1.4}), it is interesting to find conditions under which for a subgroup $H$ of $\pi_{1}(X, x_{0})$  there is an open cover $\mathcal{U}$ of $X$ such that $\pi(\mathcal{U}, x_{0}) \leq H$. Recall that Torabi et al. in \cite[Theorem 2.1]{TorabiS} proved that if $H$ is an open normal subgroup of  $ \pi_{1}^{qtop}(X,x_{0})$, then there is a Spanier group which is contained in $H$. In the following proposition, we show that if $X$ is a strong $H$-SLT spaces at $x_{0}$, then any open subgroup of $ \pi_{1}^{qtop}(X,x_{0})$ containing $H$ contains a Spanier group.

\begin{proposition}\label{2.1}
Let $H\leq \pi_{1}(X,x_{0})$ and $X$ be a connected, locally path connected and strong $H$-SLT space at $x_{0}$. If $K$ is an open subgroup of $ \pi_{1}^{qtop}(X,x_{0})$ containing $H$, then there is an open cover $\mathcal{U}$ of $X$ such that $\pi(\mathcal{U},x)\leq K$, i.e., there exists a covering map $p:\widetilde{X}\rightarrow X$ with $p_{\ast}\pi_{1}(\widetilde{X},\tilde{x}_{0})=K$.
\end{proposition}
\begin{proof}
Let $K$ be an open subgroup of $ \pi_{1}^{qtop}(X,x_{0})$ containing $H$. By the definition of the quotient topology $ \pi_{1}^{qtop}(X,x_{0})$, there is an open basis neighborhood  $W=\bigcap_{i=1}^{n}\langle I_{i}, U_{i}\rangle$ of the constant path $c_{x_{0}}$ in $\Omega(X,x_{0})=\lbrace \alpha \ \vert \ \alpha(0)=\alpha(1)=x_{0}\rbrace$ such that $W\subseteq \pi^{-1}(K)$, where $\pi: \Omega(X,x_{0}) \rightarrow \pi_{1}(X,x_{0})$ is the quotient map defined by $\pi(\alpha)=[\alpha]$. Put $U=\bigcap_{i=1}^{n} U_{i}$. By the structure of $W$, we have $x_{0}\in{U}$ and so $U\neq\emptyset$. Since $X$ is a connected, locally path connected and strong $H$-SLT space at $x_{0}$, for every $x \in{X}$ there is a path connected open neighborhood $V$ containing $x$ such that for every loop $\beta: I\rightarrow V$ based at $x$ and  for every path $\alpha: I\rightarrow X$ from $x_{0}$ to $x$, there is a loop $\lambda: I\rightarrow U$ based at $x_{0}$ such that $[\alpha \ast \beta \ast \alpha^{-1}]_{H}=[\lambda]_{H}$. Assume $\mathcal{U}$ is the open cover of $X$ consists of all neighborhoods $V$'s. We show that the Spanier group $\pi(\mathcal{U},x)$ with respect to the open cover $\mathcal{U}$ is contained in $K$. It is enough to show that this relation holds for any generator of $\pi(\mathcal{U},x)$. Let $[\alpha \ast \beta \ast \alpha^{-1}]$ be an arbitrary generator of $\pi(\mathcal{U},x)$. Since all elements of $\mathcal{U}$ are path connected, it is not hard to see that for any path $\alpha$ from $x_{0}$ to any $y\in{V}$ and for every loop $\beta$ inside $V$ based at $y$, there is a loop $\lambda$ inside $U$ such that $[\alpha \ast \beta \ast \alpha^{-1}]_{H}=[\lambda]_{H}$, that is, $[\alpha \ast \beta \ast \alpha^{-1}\ast \lambda^{-1}]\in{H}$ or equivalently, $[\alpha \ast \beta \ast \alpha^{-1}]\in H [\lambda]$. Since $Im(\lambda)\subseteq U$, we have $\lambda \in{W}$, i.e., $\pi(\lambda)=[\lambda]\in K$. On the other hand, since $H\leq K$, $[\alpha \ast \beta \ast \alpha^{-1}]\in K$. Hence the result holds.
\end{proof}

Recall that for any subgroup $H$ of a group $G$, the core of $H$ in $G$, denoted by $H_{G}$, is defined to the join of all normal subgroups of $G$ that are contained in $H$. Note that $H_{G}=\bigcap_{g\in{G}} g^{-1}Hg$ is the largest normal subgroup of $G$ which contained in $H$. By the structure of quasitopological group $ \pi_{1}^{qtop}(X,x_{0})$, if $H_{\pi_{1}(X,x_{0})}$ is open in $ \pi_{1}^{qtop}(X,x_{0})$ then so is $H$, but the converse is not true in general. Note that if the openness of $H$ in $ \pi_{1}^{qtop}(X,x_{0})$ implies the openness of $H_{\pi_{1}(X,x_{0})}$, then Theorem 3.7 of \cite{TorabiS} implies that every semicovering map is a covering map. But there is a semicovering map which is not a covering map (see \cite[Example 3.8]{BrazS}). The following corollary shows that the converse holds in relative version of strong SLT spaces at one point.

\begin{corollary}\label{2.2}
Let $X$ be a connected, locally path connected and strong H-SLT space at $x_{0}$. Then $H_{\pi_{1}(X,x_{0})}$ is open in $ \pi_{1}^{qtop}(X,x_{0})$  if and only if $H$ is an open subgroup in $ \pi_{1}^{qtop}(X,x_{0})$.
\end{corollary}

\begin{proof}
It is easy to see that if $H_{\pi_{1}(X,x_{0})}$ is open in $ \pi_{1}^{qtop}(X,x_{0})$, then so is $H$.

To prove the other direction, by Proposition \ref{2.1}, there is an open cover $\mathcal{U}$ of $X$ such that $\pi(\mathcal{U},x)\leq H$. On the other hand, since $H_{\pi_{1}(X,x_{0})}$ is the largest normal subgroup of $ \pi_{1}(X,x_{0})$ contained in $H$, we have $\pi(\mathcal{U},x)\leq H_{\pi_{1}(X,x_{0})}$. Since $\pi(\mathcal{U},x)$  is an open subgroup and $H_{\pi_{1}(X,x_{0})}$ is a subgroup in $ \pi_{1}^{qtop}(X,x_{0})$,  $H_{\pi_{1}(X,x_{0})}$  is open in $ \pi_{1}^{qtop}(X,x_{0})$.
\end{proof}

\begin{corollary}\label{2.3}
Let $H\leq \pi_{1}(X,x_{0})$ and $p:\widetilde{X}\rightarrow X$ be a semicovering map with $p_{\ast}(\widetilde{X},\tilde{x}_{0})=H$. If $X$ is a connected, locally path connected and strong $H$-SLT space at $x_{0}$, then $p$ is a covering map.
\end{corollary}

\begin{proof}
Let $p:\widetilde{X}\rightarrow X$ be a semicovering map with $p_{\ast}(\widetilde{X},\tilde{x}_{0})=H$. By \cite[Theorem 3.5]{BrazO}, $H$ is an open subgroup in $\pi_{1}^{qtop}(X,x_{0})$. Hence, Corollary \ref{2.2} implies that $H_{\pi_{1}(X,x_{0})}$  is open in $ \pi_{1}^{qtop}(X,x_{0})$. Therefore, using \cite[Theorem 3.7]{TorabiS}, $p:\widetilde{X}\rightarrow X$  is a covering map.
\end{proof}

The classification of semicovering maps were given by Brazas in \cite{BrazO} when $X$ is connected and locally path connected. By the definition of a semicovering map, it can be observed that every covering map is a semicovering map but the converse does not hold, out of semilocally simply connected spaces \cite[Example 3.8]{BrazO}. The following corollary shows that the converse does hold in a category of spaces wider than semilocally simply connected spaces which is a direct consequence of Corollary \ref{2.3}.

\begin{corollary}
Let  $X$ be a connected, locally path connected and strong SLT space at $x_{0}$. Then every semicovering map is a covering map.
\end{corollary}

It turns out that every covering map is an $\textbf{lpc}_{0}$-covering map but not vice versa. For example, the \textit{Hawiian Earring} admits generalized universal covering space \cite[Proposition 3.6]{Zastrow} but does not admit simply covering space because it is not semilocally simply connected (see \cite[Corollary 2.5.14]{Spanier}).

The following proposition provides some conditions under which any $\textbf{lpc}_{0}$-covering map is a covering map.
\begin{proposition}\label{2.5}
Let $p:\widetilde{X}\rightarrow X$ be an $\textbf{lpc}_{0}$-covering map with $p_{\ast}\pi_{1}(\widetilde{X},\tilde{x}_{0})=H\leq \pi_{1}(X,x_{0})$. Let $X$ be a strong $H$-SLT space at $x_{0}$. If $\vert p^{-1}(x_{0})\vert < \infty$, then $p$ is a covering map.
\end{proposition}

\begin{proof}
Let $p:\widetilde{X}\rightarrow X$ be an $\textbf{lpc}_{0}$-covering map with $p_{\ast}\pi_{1}(\widetilde{X},\tilde{x}_{0})=H$. In \cite[Lemma 5.10]{BrazG} it was shown that $p$ is equivalent to the endpoint projection map $p_{H}:\widetilde{X}_{H}^{wh}\rightarrow X$. Moreover, it was shown that the fiber $p^{-1}(x_{0})$ is Hausdorff and hence $(p_{H}^{-1}(x_{0}))^{wh}$ is Hausdorff (see \cite[Corollary 3.10]{Pasha}). On the other hand, since $\vert p^{-1}(x_{0})\vert < \infty$, we have $\vert p_{H}^{-1}(x_{0})\vert < \infty$. Thus, $(p_{H}^{-1}(x_{0}))^{wh}$ is discrete. It is not hard to see that $(p_{H}^{-1}(x_{0}))^{wh}$ agrees with $\frac{\pi_{1}^{wh}(X,x_{0})}{H}$ \cite[p. 246]{Pasha}. Therefore, the discreteness of $(p_{H}^{-1}(x_{0}))^{wh}$ implies that $\frac{\pi_{1}^{wh}(X,x_{0}}{H}$ is discrete, that is, $H$ is open in $\pi_{1}^{wh}(X,x_{0})$. Moreover, since $X$ is a strong $H$-SLT at $x_{0}$, by \cite[Proposition 3.7]{Pasha}, $H$ is open in $\pi_{1}^{qtop}(X,x_{0})$. Consequently, by Proposition \ref{2.1}, there exists an open cover $\mathcal{U}$ of $X$ such that $\pi(\mathcal{U},x_{0})\leq H$ which implies that $p:\widetilde{X}\rightarrow X$ is a covering map (see Proposition \ref{1.4}).
\end{proof}

In what follows in this section we investigate the relationship between some subgroups of $\pi_{1}(X,x_{0})$ in strong SLT and SLT spaces. Based on some works of \cite{Ab, Mashayekhy, Virk} there is a chain of some effective subgroups of the fundamental group $\pi_{1}(X,x_{0})$ as follows:

\begin{center}
$\pi_{1}^{s}(X,x_{0})\leq \pi_{1}^{sg}(X,x_{0})\leq \widetilde{\pi}_{1}^{sp}(X,x_{0})\leq \pi_{1}^{sp}(X,x_{0}),\ \ \ \ (\star)$
\end{center}
where $\pi_{1}^{s}(X,x_{0})$ is the subgroup of all small loops at $x_{0}$ \cite{Virk}, $\pi_{1}^{sg}(X,x_{0})$ is the subgroup of all small generated loops, i.e., the subgroup generated by the set
of $\lbrace [\alpha\ast\beta\ast\alpha^{-1}] \ \vert \ [\beta]\in{\pi_{1}^{sg}(X,\alpha(1))}\ and \ \alpha\in{P(X,x_{0})}\rbrace$. Also, $\pi_{1}^{sp}(X,x_{0})$ is the Spanier group of $X$, the intersection of the Spanier subgroups relative to open covers of $X$ \cite[Definition 2.3]{Fischerexam}, and $\widetilde{\pi}_{1}^{sp}(X,x_{0})$ is the path Spanier group, i.e., the
intersection of all path Spanier subgroups $\widetilde{\pi}(\mathcal{V}, x_{0})$ \cite[Definition 3.1]{TorabiS}, where $\mathcal{V}$ is a path open cover of $X$.

Recall that Jamali et al. in \cite[Proposition 3.2]{Jamali} proved that the equality of the first inequality of the above chain holds in SLT spaces at $x_{0}$. In the following we are going to present some conditions for the equality of the others inequalities.

\begin{theorem}\label{2.6}
Let $H$ be a subgroup of $\pi_{1}(X,x_{0})$ which is contained in $\pi_{1}^{s}(X,x_{0})$. If $X$ is an $H$-SLT space at $x_{0}$, then $\pi_{1}^{s}(X,x_{0})=\widetilde{\pi}_{1}^{sp}(X,x_{0})$.
\end{theorem}

\begin{proof}
By the chain ($\star$), it is sufficient to show that $\widetilde{\pi}_{1}^{sp}(X,x_{0})\leq\pi_{1}^{s}(X,x_{0})$. Consider $[\lambda]\in{\widetilde{\pi}_{1}^{sp}(X,x_{0})}$. We show that $\lambda$ is a small loop. Let $U$ be an open neighborhood of $X$ containing $x_{0}$. By assumption, since $X$ is an $H$-SLT space at $x_{0}$, there is a path Spanier subgroup $\widetilde{\pi}(\mathcal{V},x_{0})$ such that for $[\lambda]\in{\widetilde{\pi}_{1}^{sp}(X,x_{0})}\leq \widetilde{\pi}(\mathcal{V},x_{0})$ there is a loop $\lambda_{U}: I \rightarrow U$ based at $x_{0}$ such that $[\lambda]_{H}=[\lambda_{U}]_{H}$, i.e., $[\lambda\ast \lambda_{U}^{-1}]\in{H}$. On the other hand, since $H\leq \pi_{1}^{s}(X,x_{0})$, there is a small loop $h: I\rightarrow U$ based at $x_{0}$ such that $[\lambda\ast \lambda_{U}^{-1}]=[h]$, that is, $[\lambda]=[h\ast \lambda_{U}]$ which implies that $[\lambda]\in{\pi_{1}^{s}(X,x_{0})}$. Consequently, $\pi_{1}^{s}(X,x_{0})=\widetilde{\pi}_{1}^{sp}(X,x_{0})$.
\end{proof}



\begin{corollary}
Let $X$ be an SLT space at $x_{0}$. Then $\pi_{1}^{s}(X,x_{0})=\widetilde{\pi}_{1}^{sp}(X,x_{0})$.
\end{corollary}

\begin{theorem}
Let $H$ be a subgroup of $\pi_{1}(X,x_{0})$ which is contained in $\pi_{1}^{s}(X,x_{0})$. If $X$ is a locally path connected strong $H$-SLT space at $x_{0}$, then $\pi_{1}^{s}(X,x_{0})={\pi}_{1}^{sp}(X,x_{0})$.
\end{theorem}

\begin{proof}
The proof is similar to that of Theorem \ref{2.6}.
\end{proof}
\begin{corollary}
Let $X$ be a locally path connected strong SLT space at $x_{0}$. Then $\pi_{1}^{s}(X,x_{0})={\pi}_{1}^{sp}(X,x_{0})$.
\end{corollary}


\section{Relationship Between Open Subsets of $\pi_{1}^{wh}$ and $\pi_{1}^{l}$}

Brodskiy et al. \cite[Proposition 4.19]{BroU}) showed that there is a remarkable relation between the whisker topology and the lasso topology on $\widetilde{X}_{e}$ in strong SLT spaces. Indeed, they showed that $X$ is a strong SLT space if only if for every $x\in{X}$, $\widetilde{X}_{e}^{wh}=\widetilde{X}_{e}^{l}$. Similarly, it is of interest to determine when these topologies coincide on $\widetilde{X}_{H}$. In other words,

\begin{center}
``What is a necessary and sufficient condition for the equality $\widetilde{X}_{H}^{wh}=\widetilde{X}_{H}^{l}$?''
\end{center}

Pashaei et al. in \cite{Pasha} gave a necessary and sufficient condition for the coincidence of $\widetilde{X}_{H}^{wh}$ and $\widetilde{X}_{H}^{top}$. To answer the above question, we need the notion of strong $H$-SLT space. For a subgroup $H$ of $\pi_{1}(X, x_{0})$ we recall that $X$ is a strong $H$-SLT space if for every $x\in {X}$ and for every path $\delta$ from $x_{0}$ to $x$, $X$ is a strong $[\delta^{-1}H\delta]$-SLT space at $x$.
\begin{remark}\label{3.1}
Note that if X is a strong $H$-SLT space, then $X$ is a strong $[\delta^{-1} H\delta]$-SLT space for every path $\delta$ from $x_{0}$ to $x$. Let $X$ be a strong $H$-SLT space and $H$ be a normal subgroup of $\pi_{1}(X, x_{0})$. Consider the isomorphism $\varphi_{\delta}:\pi_{1}^{qtop}(X,x_{0}) \rightarrow \pi_{1}^{qtop}(X,x)$ defined by $\varphi_{\delta}([\beta])=[\delta^{-1}\ast\beta\ast\delta]$. Then the normality of $H$ implies that $\varphi_{\delta}(H)=[\delta^{-1}H\delta]$ is a normal subgroup of $\pi_{1}^{qtop}(X,x)$. Moreover, since $\varphi_{\delta}$ is a homeomorphism, openness of $H$ implies that $\varphi_{\delta}(H)=[\delta^{-1}H\delta]$ is open in $\pi_{1}^{qtop}(X,x)$.
\end{remark}
\begin{theorem}\label{3.2}
Let $X$ be a path connected space and $H$ be a normal subgroup of $\pi_{1}(X,x_{0})$. Then $X$ is a strong $H$-SLT space if and only if $\widetilde{X}^{l}_{[\delta^{-1}H\delta]}=\widetilde{X}^{wh}_{[\delta^{-1}H\delta]}$ for every $x\in{X}$ and for any path $\delta$ from $x_{0}$ to $x$.
\end{theorem}

\begin{proof}
Let $X$ be a strong $H$-SLT space. By definitions, it is routine to check that $\widetilde{X}^{wh}_{K}$ is finer than $\widetilde{X}^{l}_{K}$ for any subgroup $K \leq \pi_{1}(X,y)$. By Remark \ref{3.1},  it is sufficient to show that $\widetilde{X}^{l}_{H}$ is finer than $\widetilde{X}^{wh}_{H}$. Consider an open basis neighborhood $B_{H}([\alpha]_{H},U)$ in $\widetilde{X}^{wh}_{H}$, where $\alpha$ is a path from $x_{0}$ to $x$. We find an open basis neighborhood $B_{H}([\alpha]_{H}, \mathcal{U},W)$ in $\widetilde{X}^{l}_{H}$ which is contained in $B_{H}([\alpha]_{H},U)$. By the definition of strong $H$-SLT space, any point of $X$ has an open neighborhood $V$ defined by the strong $H$-SLT space property which is applied to the point $\alpha(1)=x$ and $U$, that is, for every loop $\gamma$ in $V$ and for every path $\sigma$ from $x$ to $\gamma(0)$ there is a loop $\lambda$ in $U$ based at $x$ such that $[\sigma\ast\gamma\ast\sigma^{-1}]_{[\alpha^{-1}H\alpha]}=[\lambda]_{[\alpha^{-1}H\alpha]}$. Let $\mathcal{U}$ be the open cover of $X$ consisting of all neighborhoods $V$'s. Put $W=U$  and consider $[\alpha\ast l \ast\beta]_{H}\in{B_{H}([\alpha]_{H}, \mathcal{U},U)}$. Assume $l$ is equal to a finite concatenation of loops $l=\prod_{i=1}^{n} \alpha_{i}\ast\gamma_{i}\ast\alpha^{-1}_{i}$, where  $\alpha_{i}$'s are paths from $x$ to $\alpha_{i}(1)$ and $\gamma_{i}$'s are loops in some $V\in{\mathcal{U}}$ based at $\alpha_{i}(1)$. Since $X$ is strong $H$-SLT, there are loops $\lambda_{i}$ in $U$ for $1\leq i\leq n$ such that $[\alpha_{i}\ast\gamma_{i}\ast\alpha^{-1}_{i}]_{[\alpha^{-1}H\alpha]}=[\lambda_{i}]_{[\alpha^{-1}H\alpha]}$, i.e.,  $[\alpha_{i}\ast\gamma_{i}\ast\alpha^{-1}_{i}\ast\lambda^{-1}_{i}]\in{[\alpha^{-1}H\alpha]}$ for $1\leq i\leq n$. We have $\alpha\ast\alpha_{i}\ast\gamma_{i}\ast\alpha_{i}^{-1}\ast\alpha^{-1}\simeq h_{i}$ rel $\dot{I}$, where $[h_{i}]\in{H}$ ($1\leq i\leq n$). It is enough to show that $\big[\alpha\ast\big(\prod_{i=1}^{n} \alpha_{i}\ast\gamma_{i}\ast\alpha^{-1}_{i}\big)\ast\beta\big]_{H}=\big[\alpha\ast\lambda\ast\beta\big]_{H}$, or equivalently, it is sufficient to show that $\big[\big(\prod_{i=1}^{n} \alpha\ast\alpha_{i}\ast\gamma_{i}\ast\alpha^{-1}_{i}\ast\alpha^{-1}\big)\ast \alpha\ast\lambda^{-1}\ast\alpha^{-1}\big]\in{H}$, where $\lambda=\lambda_{1}\ast\lambda_{2}\ast...\ast\lambda_{n}$ is a loop in $U$ based at $x$. Crearly, it can be seen that $\big[\big(\prod_{i=1}^{n} \alpha\ast\alpha_{i}\ast\gamma_{i}\ast\alpha^{-1}_{i}\ast\alpha^{-1}\big)\ast \alpha\ast\lambda^{-1}\ast\alpha^{-1}\big]=[(\alpha\ast\alpha_{1}\ast\gamma_{1}\ast\alpha_{1}^{-1}
\ast\lambda_{1}^{-1}\ast\alpha^{-1})\ast\alpha\ast\lambda_{1}\ast\alpha^{-1}\ast(\alpha\ast\alpha_{2}\ast\gamma_{2}\ast\alpha_{2}^{-1}\ast\lambda_{2}^{-1}
\ast\alpha^{-1})\ast\alpha\ast \lambda_{2}\ast\alpha^{-1}....\ast(\alpha\ast\alpha_{n}\ast\gamma_{n}\ast\alpha_{n}^{-1}\ast\lambda_{n}^{-1}\ast\alpha^{-1})\ast
\alpha\ast \lambda_{n}\ast\alpha^{-1}\ast(\alpha\ast\lambda_{n}^{-1}\ast\alpha^{-1})\ast(\alpha\ast\lambda_{n-1}^{-1}
\ast\alpha^{-1})\ast...\ast(\alpha\ast\lambda_{1}^{-1}\ast\alpha^{-1})]$. However, by the above observations, we have
$\big[\big(\prod_{i=1}^{n} \alpha\ast\alpha_{i}\ast\gamma_{i}\ast\alpha^{-1}_{i}\ast\alpha^{-1}\big)\ast \alpha\ast\lambda^{-1}\ast\alpha^{-1}\big]=[h_{1}\ast(\alpha\ast\lambda_{1}\ast\alpha^{-1})
\ast h_{2}\ast(\alpha\ast \lambda_{2}\ast\alpha^{-1})....\ast h_{n}\ast(\alpha\ast \lambda_{n}\ast\alpha^{-1})\ast(\alpha\ast\lambda_{n}^{-1}\ast\alpha^{-1})\ast(\alpha\ast\lambda_{n-1}^{-1}
\ast\alpha^{-1})\ast...\ast(\alpha\ast\lambda_{1}^{-1}\ast\alpha^{-1})]$. By normality of $H$, it is easy but laborious to obtain $[h]\in{H}$ such that $\big[\big(\prod_{i=1}^{n} \alpha\ast\alpha_{i}\ast\gamma_{i}\ast\alpha^{-1}_{i}\ast\alpha^{-1}\big)\ast \alpha\ast\lambda^{-1}\ast\alpha^{-1}\big]=[h_{1}\ast(\alpha\ast\gamma_{1}\ast\alpha^{-1})\ast h\ast(\alpha\ast\gamma_{1}^{-1}\ast\alpha^{-1})]$. Hence $\big[\big(\prod_{i=1}^{n} \alpha\ast\alpha_{i}\ast\gamma_{i}\ast\alpha^{-1}_{i}\ast\alpha^{-1}\big)\ast \alpha\ast\lambda^{-1}\ast\alpha^{-1}\big]=[h_{1}\ast h']$ which implies that $\big[\big(\prod_{i=1}^{n} \alpha\ast\alpha_{i}\ast\gamma_{i}\ast\alpha^{-1}_{i}\ast\alpha^{-1}\big)\ast \alpha\ast\lambda^{-1}\ast\alpha^{-1}\big]\in{H}$, where $[\alpha\ast\gamma_{1}\ast\alpha^{-1}\ast h\ast\alpha\ast\gamma_{1}^{-1}\ast\alpha^{-1}]=[h']$.

Conversely, let $x\in{X}$ and $\delta$ be a path from $x_{0}$ to $x$. It is enough to show that $X$ is a strong $[\delta^{-1}H\delta]$-SLT space at $x$. Let $U$ be an open neighborhood in $X$ containing $x$. Consider open basis neighborhood $B_{[\delta^{-1}H\delta]}\big([c_{x}]_{[\delta^{-1}H\delta]}, U\big)$ in $\widetilde{X}^{wh}_{[\delta^{-1}H\delta]}$. By $\widetilde{X}^{l}_{[\delta^{-1}H\delta]}=\widetilde{X}^{wh}_{[\delta^{-1}H\delta]}$, there is an open basis neighborhood $B_{[\delta^{-1}H\delta]}\big([c_{x}]_{[\delta^{-1}H\delta]}, \mathcal{U},W\big)$ in $\widetilde{X}^{l}_{[\delta^{-1}H\delta]}$ such that $B_{[\delta^{-1}H\delta]}\big([c_{x}]_{[\delta^{-1}H\delta]}, \mathcal{U},W\big)\subseteq B_{[\delta^{-1}H\delta]}\big([c_{x}]_{[\delta^{-1}H\delta]}, U\big)$. Pick $y\in{X}$. We know that there is an open neighborhood $V$ belonging to $\mathcal{U}$ such that $y\in{V}$. Let $\alpha$ be a path from $x$ to $y$ and $\beta$ be a loop in $V$ based at $y$. By the above relation, $[c_{x}\ast\alpha\ast\beta\ast\alpha^{-1}\ast w]_{[\delta^{-1}H\delta]}\in{B_{[\delta^{-1}H\delta]}\big([c_{x}]_{[\delta^{-1}H\delta]}, U\big)}$. Put $w=c_{x}$. Therefore, there is a loop $\lambda$ in $U$ based at $x$ such that $[\alpha\ast\beta\ast\alpha^{-1}]_{[\delta^{-1}H\delta]}=[\lambda]_{[\delta^{-1}H\delta]}$ which implies that $X$ is a strong $H$-SLT space.
\end{proof}

One can easily get the following corollary of Theorem \ref{3.2}.

\begin{corollary}\label{3.3}
Let $X$ be path connected and $H$ be any normal subgroup of $ \pi_{1}(X,x_{0}) $. Then $X$ is a strong $H$-SLT at $x_{0}$ if and only if $ (p^{-1}_H(x_0))^{wh}=(p^{-1}_H(x_0))^{l} $ or equivalently, $\frac{\pi_{1}^{wh}(X,x_{0})}{H}=\frac{\pi_{1}^{l}(X,x_{0})}{H}$.
\end{corollary}


Using Corollary \ref{3.3}, one of the main result of this section can be concluded as follows.

\begin{corollary}
Let $ H $ be any normal subgroup of $ \pi_{1}(X,x_{0})$. If $ X $ is a path connected strong $H$-SLT space at $x_{0}$, then any subset $U$ of $\pi_{1}(X,x_{0})$ containing $H$ is open in $\pi_{1}^{wh}(X,x_{0})$ if and only if it is open in $\pi_{1}^{l}(X,x_{0})$.
\end{corollary}

The following proposition shows that the property of being strong $H$-SLT is a necessary condition for the openness of the subgroup $H$ in $\pi_{1}^{l}(X,x_{0})$.

\begin{proposition}
Let $H$ be an open subgroup of $\pi_{1}^{l}(X,x_{0})$. Then $X$ is strong $H$-SLT space.
\end{proposition}

\begin{proof}
Since $H$ is an open subgroup of $\pi_{1}^{l}(X,x_{0})$, there is an open basis neighborhood $B([c_{x_{0}}], \mathcal{U}, W)$ in $\pi_{1}^{l}(X,x_{0})$ which is contained in $H$. Let $\delta$ be a path from $x_{0}$ to $x$ and $U$ be an open neighborhood containing $x$. Pick $y\in{X}$. By the definition of open cover $\mathcal{U}$, there is an open neighborhood $V$ in $\mathcal{U}$ containing $y$. Let $\alpha$ be a path from $x$ to $y$ and $\beta$ be a loop inside $V$ based at $y$. By the definition of $B([c_{x_{0}}], \mathcal{U}, W)$, we have $[c_{x_{0}}\ast\delta\ast\alpha\ast\beta\ast\alpha^{-1}\ast\delta^{-1}\ast w]\in{H}$. Put $w=c_{x_{0}}$. Hence, $[\delta\ast\alpha\ast\beta\ast\alpha^{-1}\ast\delta^{-1}]\in{H}$ or equivalently, $[\alpha\ast\beta\ast\alpha^{-1}]\in{[\delta^{-1}H\delta]}$, that is, $[\alpha\ast\beta\ast\alpha^{-1}]_{[\delta^{-1}H\delta]}=[c_{x}]_{[\delta^{-1}H\delta]}$. Therefore, $X$ is a strong $H$-SLT space.
\end{proof}
\begin{proposition}\label{3.6}
Let $H\leq \pi_{1}(X,x_{0})$ and $X$ be a locally path connected strong $H$-SLT space at $x_{0}$. Then $H$ is open in $\pi_{1}^{wh}(X,x_{0})$ if and only if $H$ is open in $\pi_{1}^{l}(X,x_{0})$.
\end{proposition}
\begin{proof}
As mentioned in the introduction, $\pi_{1}^{wh}(X,x_{0})$ is finer than $\pi_{1}^{l}(X,x_{0})$.

To prove the other direction, let $H$ be open in $\pi_{1}^{wh}(X,x_{0})$. It is easy to check that all strong $H$-SLT spaces at $x_{0}$ are $H$-SLT space at $x_{0}$. Using \cite[Proposition 3.7]{Pasha}, $H$ is open in $\pi_{1}^{qtop}(X,x_{0})$. On the other hand, by Proposition \ref{2.1}, there is an open cover $\mathcal{U}$ of $X$ such that $\pi(\mathcal{U},x_{0}) \leq H$. Choose an arbitrary element $[\alpha]\in{H}$. Consider open basis neighborhood $B([\alpha], \mathcal{U}, W)$ in $\pi_{1}^{l}(X,x_{0})$, where $W\in{\mathcal{U}}$ containing $x_{0}$. Let $[\alpha\ast l\ast w]\in{B([\alpha], \mathcal{U}, W)}$. Since $[w]\in{\pi(\mathcal{U},x_{0})}$ and  $\pi(\mathcal{U},x_{0}) \leq H$, we have $[l][w]=[l\ast w]\in{H}$. Hence, by $[\alpha]\in{H}$, we can conclude that $B([\alpha], \mathcal{U}, W)\subseteq H$. Therefore, $H$ is open in $\pi_{1}^{l}(X,x_{0})$.
\end{proof}

In the following, we show that closeness of $H$ in $\pi_{1}^{l}(X,x_{0})$ has a remarkable relation to the fibers of the endpoint projection map $p_{H}:\widetilde{X}^{l}_{H}\rightarrow X$.

\begin{proposition}\label{3.7}
Let $X$ be path connected and $H$ be a normal subgroup of $\pi_{1}(X,x_{0})$. Then $(p_{H}^{-1}(x_{0}))^{l}$ is Hausdorff if and only if $H$ is closed in $\pi_{1}^{l}(X,x_{0})$.
\end{proposition}

\begin{proof}
Let $(p_{H}^{-1}(x_{0}))^{l}$ be Hausdorff. It is sufficient to show that $\pi_{1}(X,x_{0})\setminus H$ is open in $\pi_{1}^{l}(X,x_{0})$. Let $[\alpha]\in{\pi_{1}(X,x_{0})\setminus H}$, that is, $[\alpha]\notin{H}$ or equivalently, $[\alpha]_{H}\neq [c_{x_{0}}]_{H}$. Since $(p_{H}^{-1}(x_{0}))^{l}$ is Hausdorff, there are open basis neighborhoods $M=B_{H}([\alpha]_{H}, \mathcal{U}, U)$ and $N=B_{H}([c_{x_{0}}]_{H}, \mathcal{V}, V)$ in $(p_{H}^{-1}(x_{0}))^{l}$ such that $M\cap N=\emptyset$, where $U$ and $V$ are open neighborhoods containing $x_{0}$. Consider open cover $\mathcal{W}=\lbrace U\cap V \ \vert \ U\in{\mathcal{U}} \ , \ V\in{\mathcal{V}}\rbrace$ and open basis neighborhood $B([\alpha], \mathcal{W}, U\cap V)$ containing $[\alpha]$. We show that $B([\alpha], \mathcal{W}, U\cap V)\cap H=\emptyset$. By contrary, suppose $[\alpha\ast l\ast w]\in{H}$, where $[l]\in{\pi(\mathcal{W},x_{0})}$ and $w$ is a loop inside $U\cap V$ based at $x_{0}$. Note that $[\alpha\ast l\ast w]\in{H}$ is equivalent to $[\alpha\ast l\ast w]_{H}=[c_{x_{0}}]_{H}$, i.e.,  $[\alpha\ast l\ast w]_{H}\in{M\cap N}$ which is a contradiction.

Conversely, assume $[\alpha]_{H}\neq [\beta]_{H}$, that is, $[\alpha\ast \beta^{-1}]\notin{H}$, where $[\alpha]_{H}$, $[\beta]_{H}\in{(p_{H}^{-1}(x_{0}))^{l}}$. Note that since $H$ is a normal subgroup of $\pi_{1}(X,x_{0})$, it is easy to see that $[\alpha^{-1}\ast\beta]\notin{H}$. Since $H$ is a closed subgroup of $\pi_{1}^{l}(X,x_{0})$, there is  an open cover $\mathcal{U}$ of $X$ such that $B([\alpha^{-1}\ast\beta],\mathcal{U}, U)\cap H=\emptyset$, where $U\in{\mathcal{U}}$. Consider open basis neighborhoods $M=B_{H}([\beta]_{H},\mathcal{U}, U)$ and $N=B_{H}([\alpha]_{H},\mathcal{U}, U)$. It is enough to show that $M\cap N=\emptyset$. By contrary, suppose $[\gamma]_{H}\in{M\cap N}$. So, there are $[l]$, $[s]\in{\pi(\mathcal{U}, x_{0})}$ and loops $w$ and $v$ inside $U$ based at $x_{0}$ such that $[\gamma]_{H}=[\beta\ast l\ast w]_{H}=[\alpha\ast s \ast v]_{H}$, that is, $[\beta\ast l\ast w \ast v^{-1} \ast s^{-1} \ast \alpha]\in{H}$. By the normality of $H$, we have $[\alpha^{-1}\ast\beta\ast l\ast w \ast v^{-1} \ast s^{-1}]\in{H}$ or equivalently, $[\alpha^{-1}\ast\beta\ast l\ast w \ast v^{-1} \ast s^{-1}\ast c_{x_{0}}]\in{H}$, where $[l\ast w \ast v^{-1} \ast s^{-1}]\in{\pi(\mathcal{U},x_{0})}$ and the constant path $c_{x_{0}}$ can be seen as loop inside $U$ based at $x_{0}$. This is a contradiction because $B([\alpha^{-1}\ast\beta],\mathcal{U}, U)\cap H=\emptyset$.
\end{proof}

\begin{corollary}\label{3.8}
Let $H$ be any normal subgroup of $\pi_{1}(X,x_{0})$ and $X$ be a locally path connected strong $H$-SLT space at $x_{0}$. Then $H$ is closed in $\pi_{1}^{wh}(X,x_{0})$ if and only if $H$ is closed in $\pi_{1}^{l}(X,x_{0})$.
\end{corollary}

\begin{proof}
Note that $\pi_{1}^{wh}(X,x_{0})$ is finer than $\pi_{1}^{l}(X,x_{0})$, in general \cite{VZcom}.

To prove the other direction, let $H$ be closed in $\pi_{1}^{wh}(X,x_{0})$. We know that all strong $H$-SLT spaces at $x_{0}$ are $H$-SLT space at $x_{0}$. Hence by \cite[Corollary 2.8]{Pasha}, $H$ is closed in $\pi_{1}^{qtop}(X,x_{0})$. On the other hand, by \cite[Theorem 11]{BrazFa}, $p_{H}:\widetilde{X}^{wh}_{H}\rightarrow X$ has unique path lifting property, that is, $p_{H}$ is $\textbf{lpc}_{0}$-covering map (see \cite[Theorem 5.11]{BrazG}). Moreover, by \cite[Corollary 3.10]{Pasha}, $(p_{H}^{-1}(x_{0}))^{wh}$ is Hausdorff. Therefore,  Corollary \ref{3.3} and Proposition \ref{3.7} imply that $(p_{H}^{-1}(x_{0}))^{l}$ is Hausdorff and accordingly, $H$ is closed in $\pi_{1}^{l}(X,x_{0})$.
\end{proof}

 In \cite[Lemma 5.10]{BrazG} it was shown that every $\textbf{lpc}_{0}$-covering map is equivalent to a certain endpoint projection map. On the other hand, the unique path lifting property of the endpoint projection map $p_{H}:\widetilde{X}^{wh}_{H}\rightarrow X$ implies that $p_{H}$ is an $\textbf{lpc}_{0}$-covering map \cite[Theorem 5.11]{BrazG}. The main advantage of the following theorem is that the map $p:\widetilde{X}\rightarrow X$ with $p_{\ast}\pi_{1}(\widetilde{X}, \tilde{x}_{0})=H\leq \pi_{1}(X,x_{0})$ is an $\textbf{lpc}_{0}$-covering map when $p_{H}:\widetilde{X}^{l}_{H}\rightarrow X$ has unique path lifting property.

\begin{proposition}
Let $H\unlhd \pi_{1}(X,x_{0})$. If $X$ is a locally path connected and strong $H$-SLT space at $x_{0}$, then the following statements are equivalent.
\begin{itemize}
\item[1].
$p_{H}:\widetilde{X}^{wh}_{H}\rightarrow X$ has unique path lifting property.

\item[2].
$p_{H}:\widetilde{X}^{l}_{H}\rightarrow X$ has unique path lifting property.
\end{itemize}
\end{proposition}

\begin{proof}
$1.\Rightarrow 2.$ It follows from the statement ``$\widetilde{X}^{wh}_{H}$ is finer than $\widetilde{X}^{l}_{H}$''.

$2.\Rightarrow 1.$ The unique path lifting property of $p_{H}:\widetilde{X}^{l}_{H}\rightarrow X$ is equivalent to the closeness of $H$ in $\pi_{1}^{l}(X,x_{0})$ (see \cite[Theorem 5.6]{Bcovering}). So, by Corollary \ref{3.8}, $H$ is closed in $\pi_{1}^{wh}(X,x_{0})$.Therefore, using \cite[Corollary 2.8]{Pasha}, $H$ is closed in $\pi_{1}^{qtop}(X,x_{0})$ which implies that $p_{H}:\widetilde{X}^{wh}_{H}\rightarrow X$ has unique path lifting property \cite[Theorem 11]{BrazFa}.
\end{proof}

\begin{proposition}
Let $p:\widetilde{X}\rightarrow X$ be a semicovering map and $\widetilde{H}$ be a subgroup of $\pi_{1}(\widetilde{X},\tilde{x}_{0})$, where $p(\tilde{x}_{0})=x_{0}$. Let $H=p_{\ast}(\widetilde{H})$ where $p_{\ast}:\pi_{1}(\widetilde{X},\tilde{x}_{0})\rightarrow \pi_{1}(X,x_{0})$ is the induced homomorphism by $p$. If $X$ is a strong $H$-SLT space, then $\widetilde{X}$ is a strong $\widetilde{H}$-SLT space.
\end{proposition}

\begin{proof}
Assume $\tilde{\lambda}$ is an arbitrary path from $\tilde{x}_{0}$ to $\tilde{\lambda}(1)=\tilde{x}$. It suffices to show that $X$ is a strong $[(\tilde{\lambda})^{-1}H \tilde{\lambda}]$-SLT space at $\tilde{x}$. Let $\widetilde{S}$ be an open subset in $\widetilde{X}$ containing $\tilde{x}$ and $\tilde{y}$ be an arbitrary point in $\widetilde{X}$. Put $p\circ\tilde{\lambda}=\lambda$ and $p(\tilde{x})=x$, where $\lambda$ is a path in $X$ from $x_{0}$ to $x$. Since $p:\widetilde{X}\rightarrow X$ is a local homeomorphism, there is an open subset $\widetilde{W}$ of $\tilde{x}$ such that $p\vert_{\widetilde{W}}: \widetilde{W}\rightarrow W$ is a homeomorphism. Put $\widetilde{U}=\widetilde{S}\cap \widetilde{W}$. Note that $p\vert_{\widetilde{U}}: \widetilde{U}\rightarrow U$ is a homeomorphism as well, where $p(\widetilde{U})=U$ is an open subset of $x$ in $X$. By assumption, since $X$ is a strong $[\lambda^{-1}H\lambda]$-SLT space at $x$, for point $p(\tilde{y})=y$ there is an open subset $V$ containing $y$ such that for every path $\alpha$ from $x$ to $y$ and for every loop $\beta$ at $y$ in $V$ there is a loop $\delta$ based at $x$ in $U$ such that $[\alpha\ast\beta\ast\alpha^{-1}]_{[\lambda^{-1}H\lambda]}=[\delta]_{[\lambda^{-1}H\lambda]}$. By the local homeomorphism property of $p:\widetilde{X}\rightarrow X$, we have an open subset $\widetilde{V}$ of $\tilde{y}$ such that $p\vert_{\widetilde{V}}: \widetilde{V}\rightarrow V$ is a homeomorphism. Now let $\tilde{\alpha}$ be a path from $\tilde{x}$ to $\tilde{y}$ and $\tilde{\beta}$ is a loop based at $\tilde{y}$ in $\widetilde{V}$. Put $p\circ\tilde{\alpha}=\alpha$ and $p\circ\tilde{\beta}=\beta$, where $\alpha$ is a path from $x$ to $y$  and $\beta$ is a loop at $y$ in $V$. Hence there is a loop $\delta:I\rightarrow U$ at $x$ such that $[\alpha\ast \beta\ast\alpha^{-1}]_{[\lambda^{-1}H \lambda]}=[\delta]_{[\lambda^{-1}H \lambda]}$ or equivalently, $[\lambda\ast\alpha\ast\beta\ast\alpha^{-1}\delta^{-1}\ast\lambda^{-1}]\in{H}$. By the homeomorphism $p\vert_{\widetilde{U}}: \widetilde{U}\rightarrow U$, there is a loop $\tilde{\delta}$ at $\tilde{x}$ in $\widetilde{U}$ such that $p(\tilde{\delta})=\delta$. On the other hand, we know that $p_{\ast}([\tilde{\lambda}\ast\tilde{\alpha}\ast\tilde{\beta}\ast(\tilde{\alpha})^{-1}\ast(\tilde{\delta})^{-1}\ast
(\tilde{\lambda})^{-1}])=[(p\circ\tilde{\lambda})\ast(p\circ\tilde{\alpha})\ast(p\circ\tilde{\beta})\ast
(p\circ(\tilde{\alpha})^{-1})\ast(p\circ(\tilde{\delta})^{-1})\ast(p\circ(\tilde{\lambda})^{-1})]=[\lambda
\ast\alpha\ast\beta\ast\alpha^{-1}\delta^{-1}\ast\lambda^{-1}]\in{H}$. Moreover, by the definition of semicovering map \cite[Definition 3.1]{BrazS}, $p_{\ast}$ is injection. Thus, by the definition of $\widetilde{H}$, we have $[\tilde{\lambda}\ast\tilde{\alpha}\ast\tilde{\beta}\ast(\tilde{\alpha})^{-1}\ast(\tilde{\delta})^{-1}\ast
(\tilde{\lambda})^{-1}]\in{\widetilde{H}}$ or equivalently, $[\tilde{\alpha}\ast\tilde{\beta}\ast(\tilde{\alpha})^{-1}]_{[(\tilde{\lambda})^{-1}H \tilde{\lambda}]}=[(\tilde{\delta})^{-1}]_{[(\tilde{\lambda})^{-1}H \tilde{\lambda}]}$. Therefore, $\widetilde{X}$ is a strong $\widetilde{H}$-SLT space.
\end{proof}

\begin{corollary}\label{3.11}
Let $p:\widetilde{X}\rightarrow X$ be a semicovering map. If $X$ is a strong SLT space, then so is $\widetilde{X}$.
\end{corollary}

It is known that any semilocally simply connected space is a strong SLT space and any strong SLT space is a strong $H$-SLT space for every subgroup $H$ of $\pi_{1}(X,x_{0})$. Note that any space $X$ is a strong $\pi_{1}(X,x_{0})$-SLT space. The following theorem help us to give an example of a strong $H$-SLT space which is not a strong SLT space and hence, it is not semilocally simply connected, where $H\neq \pi_{1}(X,x_{0})$.

\begin{theorem}\label{3.12}
Let $X$ be a locally path connected space and $H$ be an open normal subgroup in $\pi_{1}^{qtop}(X,x_{0})$. Then $X$ is a strong $H$-SLT space.
\end{theorem}

\begin{proof}
Let $H$ be an open normal subgroup in $\pi_{1}^{qtop}(X,x_{0})$. By Theorem 2.1 of \cite{TorabiS}, There is an open cover $\mathcal{U}$ in $X$ such that $\pi(\mathcal{U}, x_{0})\leq H$. On the other hand, by \cite[Proposition 4.4]{Bcovering}, $\widetilde{X}^{wh}_{H}=\widetilde{X}^{l}_{H}$. So, using Remark \ref{3.1}, we can conclude that for each path $\delta$ from $x_{0}$ to $x$, $\widetilde{X}^{l}_{[\delta^{-1}H\delta]}=\widetilde{X}^{wh}_{[\delta^{-1}H\delta]}$. Therefore, Theorem \ref{3.2} implies that $X$ is a strong $H$-SLT space.
\end{proof}
The following example can justify introducing the relative version of strong SLT spaces with respect to subgroups of the fundamental group.
\begin{example}\label{3.13}
Let $(S^{1},0)$ be the unit circle, $(HA,x)$ be the Harmonic Archipelago, where $x$ is the common point of boundary circles. We consider the wedge space $X=\frac{S^{1} \sqcup HA}{0\sim x}$. In \cite[Example 4.4]{Torabi} it is shown that $\pi_{1}(X,x_{0})\neq \pi_{1}^{sg}(X,x_{0})$. On the other hand, $X$ is a semilocally small generated space \cite{Torabi}. Accordingly, $\pi_{1}^{sg}(X,x_{0})$, introduced by Virk \cite{Virk}, is an open subgroup of $\pi_{1}^{qtop}(X,x_{0})$. Using Theorem \ref{3.12}, we conclude that $X$ is a strong $\pi_{1}^{sg}(X,x_{0})$-SLT space. It is not hard to show that $X$ is not a strong SLT space. To prove that $X$ is not a strong SLT space, consider an arbitrary path in $X$ inside $HA$ from any semilocally simply connected point to the wedge point.
\end{example}

In the following example we show that some results of the paper does not necessarily hold, for instance Proposition \ref{3.7}, for $H$-SLT spaces at a point.
 Also, note that the example below shows that the concepts of relative version of strong SLT and SLT spaces are not necessarily the same.

\begin{example}\label{3.14}
In \cite{FischerC}, Fischer and Zastrow presented an open subgroup $H$ of $\pi_{1}^{qtop}(HE,x_{0})$ which does not contain any open normal subgroup and does not correspond to a covering space \cite[Theorem 4.8]{BrazO}, where $x_{0}$ is the common point of circles. On the other hand, by Definition \ref{1.2} and Proposition \ref{1.4}, it is not hard to observe that covering subgroups correspond to open subgroups of $\pi_{1}^{l}(X,x_{0})$. Moreover, since $H$ is an open subgroup of $\pi_{1}^{qtop}(HE,x_{0})$, Proposition 3.6 of \cite{Pasha} implies that $HE$ is an $H$-SLT space at $x_{0}$ and hence, using \cite[Proposition 3.7]{Pasha}, $H$ is an open subgroup of $\pi_{1}^{wh}(X,x_{0})$. One can see that $H$ is not an open subgroup in $\pi_{1}^{l}(HE,x_{0})$ because it is not a covering subgroup. Therefore, the property of $H$-SLT at $x_{0}$ is not strong enough to prove some results of this paper. Also, we can conclude that $HE$ is an $H$-SLT space which is not a strong $H$-SLT space.
\end{example}


\section*{Reference}

\bibliography{mybibfile}




\end{document}